\documentclass[11pt,a4paper]{amsart}
\usepackage{enumerate}
\usepackage{amsfonts,amssymb}
\usepackage{amsthm,amsmath,amstext,xfrac}
\usepackage{cancel,xspace}
\usepackage{epsfig,fancybox}
\usepackage{graphicx,mathrsfs,ifpdf}
\usepackage{setspace}

\newtheorem{theorem}{Theorem}[section]
\newtheorem{lemma}[theorem]{Lemma}

\newtheorem{proposition}[theorem]{Proposition}

\newtheorem{problem}[theorem]{Problem}
\newtheorem{conjecture}[theorem]{Conjecture}

\theoremstyle{definition}

\theoremstyle{theorem}

\numberwithin{equation}{section}
\numberwithin{figure}{section}

\newcommand\Set[2] {\left\{{#1}:\,{#2}\right\}}
\newcommand\Setx[1] {\left\{{#1}\right\}}
\newcommand{\cp}[2]{({#1},{#2})}
\newcommand{\cpx}[2]{{#1},{#2}}

\newcommand{\RR}{\mathbb R}
\newcommand{\ZZ}{\mathbb Z}

\DeclareMathOperator{\sd}{sd}
\DeclareMathOperator{\CN}{CN}

\DeclareMathOperator{\id}{id}
\DeclareMathOperator{\Img}{Im}

\DeclareMathOperator{\coind}{coind}

\newcommand{\fig}[1]{\includegraphics[page=#1]{projective-figs.pdf}}%
\newcommand{\hf}{\hspace{0mm}\hspace*{\fill}\hspace{0mm}}
\newcommand{\claimproofend}{\hspace*{.1mm}\hspace{\fill}}

\begin{document}
%\onehalfspacing
\title{Colouring quadrangulations of projective spaces}
\author{Tom\'a\v s Kaiser}\thanks{The Institute for Theoretical
  Computer Science (CE-ITI) is supported by project P202/12/G061 of
  the Czech Science Foundation.}
  % The first author was supported by project GA14-19503S of
  % the Czech Science Foundation.

\address{Department of Mathematics, Institute for Theoretical Computer
  Science (CE-ITI), and European Centre of Excellence NTIS (New
  Technologies for the Information Society), University of West
  Bohemia, Univerzitn\'i 8, 306 14 Plze\v n, Czech Republic}
\email{kaisert@kma.zcu.cz}
\author{Mat\v ej Stehl\'ik}
\thanks{The second author was partially supported by ANR project Stint
  under reference ANR-13-BS02-0007 and by the LabEx PERSYVAL-Lab (ANR--11-LABX-0025)}
\address{UJF-Grenoble 1 / CNRS / Grenoble-INP, G-SCOP UMR5272 Grenoble, F-38031, France}
\email{matej.stehlik@g-scop.inpg.fr}
%\date{\today}

\begin{abstract}
  A graph embedded in a surface with all faces of size $4$ is known as
  a quadrangulation. We extend the definition of quadrangulation to
  higher dimensions, and prove that any graph $G$ which embeds as a
  quadrangulation in the real projective space $P^n$ has chromatic
  number $n+2$ or higher, unless $G$ is bipartite. For $n=2$ this
  was proved by Youngs [J. Graph Theory 21 (1996), 219--227]. The
  family of quadrangulations of projective spaces includes all
  complete graphs, all Mycielski graphs, and certain graphs homomorphic to
  Schrijver graphs. As a corollary, we obtain a new proof of the
  Lov\'asz--Kneser theorem.
\end{abstract}
\maketitle

\section{Introduction}
\label{sec:introduction}

A graph which embeds in the real projective plane $P^2$ so that every
face is bounded by a walk of length $4$ is called a ($2$-dimensional) \emph{projective
  quadrangulation}. A remarkable result of Youngs~\cite{You96} asserts
that the chromatic number of a projective quadrangulation is either
$2$ or $4$. On the last page of \cite{You96}, Youngs notes:
\begin{quote}\small
  \dots it would be equally worthwhile to increase the chromatic number [of the
  graphs in question]. A possible step in this direction is to jump
  from a ($2$-dimensional) projective plane to a higher dimensional
  projective space. This may not be a fruitful path to follow, and the
  only evidence the author can suggest in its favor is that the
  $5$-chromatic Mycielski graphs embed pleasantly in projective $3$-space
  in a similar fashion to their $4$-chromatic counterparts in $2$-space.
\end{quote}

In this paper, we show that Youngs' intuition was correct as we extend
the lower bound in his theorem to the $n$-dimensional real projective space
$P^n$. To do so, we extend the notion of quadrangulation to higher dimensions
as follows (for definitions, see Section~\ref{sec:terminology}).

Let $K$ be a generalised simplicial complex (there may be more than one simplex
with the same set of vertices, unlike in the usual simplicial complex). A \emph{quadrangulation}
of $K$ is a spanning subgraph $G$ of its $1$-skeleton $K^{(1)}$ such that every
(inclusionwise) maximal simplex of $K$ induces a complete bipartite
subgraph of $G$ with at least one edge. If the polyhedron of $K$ is
homeomorphic to a topological space $X$, we say that the natural
embedding of $G$ in $X$ is a quadrangulation of $X$. 

Note that if $K$ triangulates the projective plane, then a
quadrangulation of $K$ is a projective quadrangulation according to
the usual definition recalled at the beginning of this
section. Conversely, given a projective quadrangulation $H$, we can
triangulate its faces and obtain $H$ as a quadrangulation of the
resulting generalised simplicial complex. More precisely, this is true
if none of the faces of $H$ contains a crosscap; otherwise, two edges
of $H$ will be doubled in the process. However, this difference
between the two definitions is unimportant as long as we are
interested in vertex colouring.

Our main result is the following generalisation of the lower bound of
Youngs.
\begin{theorem}
\label{thm:main}
  If $G$ is a non-bipartite quadrangulation of the $n$-dimensional
  projective space $P^n$, then $\chi(G) \geq n+2$.
\end{theorem}

We show that the family of quadrangulations of projective spaces
includes all complete graphs and all (generalised) Mycielski graphs.
We also prove the following result about the Schrijver graph
$SG(n,k)$. (Recall that a graph $G$ is \emph{homomorphic} to a graph
$H$ if there exists a mapping $f:V(G)\to V(H)$ such that
$f(u)f(v) \in E(H)$ whenever $uv \in E(G)$; note that, in this case,
$\chi(G) \leq \chi(H)$.)
\begin{theorem}\label{thm:schrijver}
  Let $n > 2k$ and $k \geq 1$. There exists a non-bipartite
  quadrangulation of $P^{n-2k}$ that is homomorphic to $SG(n,k)$.
\end{theorem}
Since the Schrijver graph $SG(n,k)$ is a subgraph of the Kneser graph
$KG(n,k)$, Theorems~\ref{thm:main} and~\ref{thm:schrijver}
give an alternative proof of the Lov\'asz--Kneser Theorem~\cite{Lov78},
namely $\chi(KG(n,k)) \geq n-2k+2$.

It may come as a surprise that the chromatic number of
quadrangulations of $P^n$ cannot be bounded from above for any $n>2$,
as the next theorem shows.
\begin{theorem}
\label{thm:upper}
  For $n \geq 3$ and $t\geq 5$, the complete graph $K_t$ embeds in
  $P^n$ as a quadrangulation if $t-n$ is even.
\end{theorem}
However, we show that sufficiently `fine' non-bipartite quadrangulations
of $P^n$ are $(n+2)$-chromatic.

The rest of the paper is organised as follows. In
Section~\ref{sec:terminology} we introduce the basic terminology and
preliminary results. In Section~\ref{sec:quadrangulations} we prove
two simple lemmas about quadrangulations which will be used later.
In Section~\ref{sec:lower} we prove
Theorem~\ref{thm:main}, and briefly discuss bounds on variants of the
chromatic number. In Section~\ref{sec:upper}, we prove
Theorem~\ref{thm:upper} and provide a geometric sufficient condition
for $(n+2)$-colourability. In Section~\ref{sec:mycielski} we show how
complete graphs and Mycielski graphs embed as quadrangulations in $P^n$,
and use it to prove Theorem~\ref{thm:schrijver}. We conclude by
presenting a conjecture and two open problems in
Section~\ref{sec:conclusion}.

\section{Topological preliminaries}
\label{sec:terminology}

Our graph theory terminology and notation is standard, and follows Bondy
and Murty~\cite{BonMur08}. All graphs considered are simple, that is, have
no loops and multiple edges. The vertex and edge set of a graph $G$ is
denoted by $V(G)$ and $E(G)$, respectively.

For a comprehensive account of topological methods in combinatorics and
graph theory, we refer the reader to Matou\v{s}ek~\cite{Mat03} or
Kozlov~\cite{Koz08}. For an introduction to algebraic topology, see
Hatcher~\cite{Hat02} or Munkres~\cite{Mun84}.

We will deal with several different kinds of simplicial complexes. By
default, our complexes are generalised simplicial
complexes~\cite[Section~2.2]{Koz08} (also known as regular
$\Delta$-complexes~\cite{Hat02} or simplicial cell complexes). A
topological space $K$ (a subspace of some Euclidean space $\RR^N$) is
a \emph{generalised simplicial complex} if it can be constructed
inductively using the following `gluing process'. We start with a
discrete point space $K^{(0)}$ in $\RR^N$, and at each step $i>0$ we
inductively construct the space $K^{(i)}$ by attaching a set of
$i$-dimensional simplices to $K^{(i-1)}$. We call the images of the
simplices involved in the construction the \emph{faces} or
\emph{cells} of $K$. Each simplex is attached via a gluing map
$f:\partial \Delta_i \to K^{(i-1)}$ that maps the interior of each
face of the boundary of the standard $i$-simplex $\Delta_i$ in
$\RR^{i+1}$ homeomorphically to the interior of a face of $K^{(i-1)}$
of the same dimension. For each $i$, the set $K^{(i)}$ is called the
\emph{$i$-skeleton} of $K$. The set of \emph{vertices} $K^{(0)}$ is
also denoted by $V(K)$. All the generalised simplicial complexes in
this paper have a finite number of faces.

The \emph{polyhedron} $\|K\|$ of a generalised simplicial complex $K$
is defined as the union of all of its cells (in $\RR^N$). We say that
$K$ \emph{triangulates} the space $\|K\|$ or any space homeomorphic to
it. All triangulations will be generalised simplicial complexes unless
otherwise noted.

A generalised simplicial complex $K$ is a \emph{geometric simplicial
  complex} if the embedding of each face is a linear map (a linear
extension of the embedding of $K^{(0)}$).

An \emph{abstract simplicial complex} is a non-empty hereditary set
system. Given a generalised simplicial complex $K$, a natural way to
assign an abstract simplicial complex to it is as follows. Let $A(K)$
be the multiset of the vertex sets of all the faces of $K$. If $A(K)$
is actually a set (that is, all the faces have distinct vertex sets),
then it is an abstract simplicial complex and we say that $K$ is a
\emph{realisation} of $A(K)$ (or a \emph{geometric realisation} if $K$
is a geometric simplicial complex). Furthermore, we say that $\|K\|$
is the \emph{polyhedron} of $A(K)$. It is well known that every finite
abstract simplicial complex has a geometric realisation, and the
polyhedra of all of its realisations are homeomorphic.

If $K$ is a generalised simplicial complex in $\RR^N$ such that for
any face $\sigma$ of $K$, its central reflection $-\sigma$ is also a
face of $K$, then we say that $K$ is an \emph{(antipodally) symmetric
  triangulation} of $\|K\|$. Of particular importance for us will be
symmetric triangulations of the unit sphere $S^n$. Furthermore, if $K$
triangulates the unit $n$-ball $B^n$ and the subcomplex corresponding
to the boundary $\partial B^n=S^{n-1}$ is a symmetric triangulation of
$S^{n-1}$, then we say that $K$ is a \emph{boundary-symmetric}
triangulation of $B^n$.

The proof of Theorem~\ref{thm:main} relies on a special type of
abstract simplicial complex associated to a graph, which we shall now
define. Given a graph $G$, the set of \emph{common neighbours} of a
set $A \subseteq V(G)$ is defined as
\[
\CN(A)=\Set{v\in V(G)}{\{a,v\}\in E(G) \text{ for all } a\in A}.
\]
The \emph{box complex} of a graph $G$ without isolated vertices is the
simplicial complex with vertex set $V(G) \times \Setx{1,2}$, defined as
\[
B(G)=\Set{A_1 \uplus A_2}{A_1,A_2 \subseteq V(G),
A_1 \subseteq \CN(A_2)\neq \emptyset,
A_2 \subseteq \CN(A_1)\neq \emptyset},
\]
where we use the notation $A \uplus B$ for the set
$(A\times\Setx{1})\cup (B\times \Setx{2})$.

Let $K$ be a generalised simplicial complex and $p$ a non-negative integer. Restricting
to $\ZZ_2$ coefficients, recall that a \emph{$p$-chain} of $K$ is a
(finite) formal sum of some of the $p$-simplices of $K$, and the group
of $p$-chains of $K$ is denoted by $C_p(K,\mathbb Z_2)$. The
\emph{boundary} of a $p$-chain $c$ is denoted by $\partial_p(c)$. The
group of \emph{$p$-cycles} is defined as $Z_p(K,\mathbb
Z_2)=\ker \partial_p$ and the group of \emph{$p$-boundaries} of
$C_p(K,\mathbb Z_2)$ as $B_p(K,\mathbb Z_2)=\Img \partial_p$. The
\emph{$p$-th homology} group is the quotient $H_p(K,\mathbb
Z_2)=Z_p(K,\mathbb Z_2)/B_p(K,\mathbb Z_2)$. Two $p$-cycles $c_1, c_2
\in Z_p(K,\mathbb Z_2)$ are \emph{homologous} if they are in the same
class of $H_p(K,\mathbb Z_2)$, i.e., if there exists a $(p+1)$-chain
$d$ such that $c_1+c_2=\partial_{p+1}(d)$.

A homeomorphism $\nu:X \to
X$ is called a \emph{$\mathbb Z_2$-action} on $X$ if $\nu^2 = \nu \circ
\nu=\id_X$. The $\mathbb Z_2$-action $\nu$ is \emph{free} if it has no
fixed points. A topological space $X$ equipped with a (free)
$\mathbb Z_2$-action $\nu$ is a \emph{(free) $\mathbb Z_2$-space}.
A canonical example of a free $\mathbb Z_2$-space is $(S^n,\nu)$, where
$S^n$ is the unit $n$-sphere and $\nu$ is the \emph{antipodal action}
given by $\nu:x \mapsto -x$. The box complex is equipped with
a natural free $\mathbb Z_2$-action $\nu$ which interchanges the two
copies of $V(G)$, namely $\nu:(v,1) \mapsto (v,2)$ and $\nu:(v,2)
\mapsto (v,1)$.
Given $\mathbb Z_2$-spaces $(X,\nu)$ and $(Y,\omega)$, a continuous map
$f:X \to Y$ such that $f \circ \nu = \omega \circ f$ is known
as a \emph{$\mathbb Z_2$-map}. If there exists a $\mathbb Z_2$-map from
$X$ to $Y$, we write $X \xrightarrow{\mathbb Z_2} Y$. The \emph{$\mathbb
Z_2$-coindex} of $X$ is defined as
\[
\coind(X)=\max \Set{n \geq 0}{S^n \xrightarrow{\mathbb Z_2} X}.
\]
When $K$ is a generalised simplicial complex, we write
$\coind(K)$ instead of $\coind(\|K\|)$.

The main tool in the proof of Theorem~\ref{thm:main} is the following
inequality, which may be traced to Lov\'asz's seminal paper~\cite{Lov78}
(where it was stated in terms of the connectivity of the neighbourhood
complex). Its proof relies on the Borsuk-Ulam theorem~\cite{Bor33}. (See
also Theorem~5.9.3 in~\cite{Mat03} and the discussion on page~99 therein.)

\begin{theorem}
  \label{thm:lovasz}
  If $G$ is a graph without isolated vertices, then $\chi(G) \geq \coind(B(G))+2$.
\end{theorem}

\section{Quadrangulations}
\label{sec:quadrangulations}

Recall the definition of a quadrangulation of $X$ from the introduction,
namely it is a subgraph $G$ of the $1$-skeleton $K^{(1)}$ of a
generalised simplicial complex $K$ such that $\|K\| \cong X$, and
every maximal simplex of $K$ induces a complete bipartite subgraph of
$G$ with at least one edge.

We define the \emph{parity} of a cycle in a graph to be the
parity of its length; a cycle is even (resp.\ odd) if it has even
(resp.\ odd) parity. We start by a proving the following crucial
property of quadrangulations of projective spaces.

\begin{lemma}
\label{lem:parity}
In every quadrangulation $G$ of a topological space $X$, homologous
cycles have the same parity; in particular, $0$-homologous cycles are
even. If $X=P^n$ and $G$ is not bipartite, then every
$1$-homologous cycle is odd.
\end{lemma}

\begin{proof}
  Let $K$ be a generalised simplicial complex whose polyhedron is
  homeomorphic to $X$, such that $G$ is a quadrangulation of $K$.  If
  $a$ and $b$ are homologous $1$-cycles in $G$, we may write
  $a+b=\partial_2(c)$, for some $2$-chain $c \in C_2(K,\mathbb
  Z_2)$. By the definition of quadrangulation, every $2$-simplex in
  $K$ is incident with an even number of edges of $G$ (namely $0$ or
  $2$), so the boundary $\partial_2(c)$ is incident with an even
  number of edges of $G$. Hence, the parity of the length of $a$ and
  $b$ is the same. It also follows that $0$-homologous cycles are
  even.

  We prove the last assertion of the lemma.  It is well known
  (cf.~\cite{Mun84}) that $H_1(P^n,\mathbb Z_2) \cong \mathbb Z_2$, so
  there are only two $\mathbb Z_2$-homology classes of cycles in
  $P^n$: the $0$- and $1$-homologous cycles, which correspond to the
  contractible and non-contractible cycles, respectively.  If $G$ is a
  non-bipartite quadrangulation of $P^n$, it contains at least one odd
  cycle, which must be $1$-homologous. Hence, by the first part of the
  lemma, every $1$-homologous cycle is odd.
\end{proof}

A $2$-colouring $c$ of a complex $K$ is an arbitrary assignment of two
colours, say black and white, to the vertices of $K$. We say that $c$ is
\emph{proper} if there is no monochromatic maximal simplex. The graph
\emph{associated} to the $2$-colouring is a spanning subgraph of the
$1$-skeleton of $K$ consisting of all edges with one end white and the
other black.

The following lemma (which will also be needed later on) makes it
easier to draw examples of quadrangulations using triangulations of
$B^n$ symmetric on the boundary. Figure~\ref{fig:complete} shows,
using condition (c), how the complete graphs $K_3$, $K_4$ and $K_5$
embed as quadrangulations in $P^1$, $P^2$ and $P^3$, respectively.

\begin{figure}
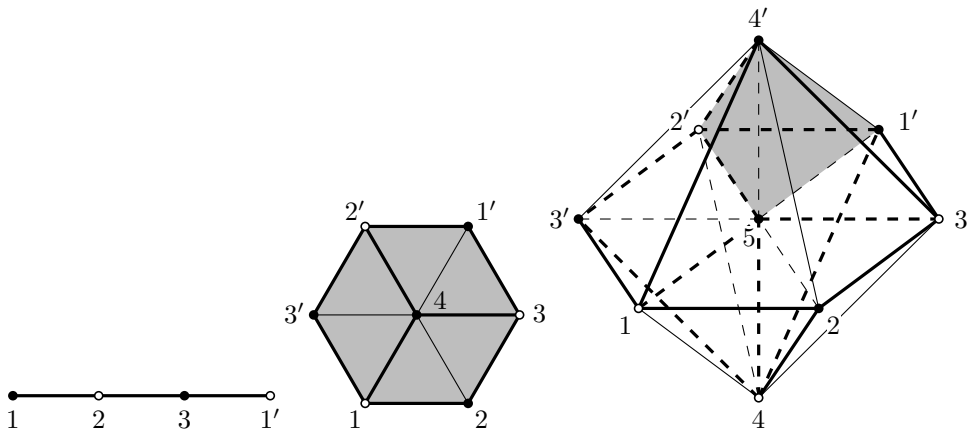

  \centering
  \fig{9}\hf\fig{10}\hf\fig{11}
  \caption{Complete graphs $K_n$ ($3\leq n \leq 5$) as
    quadrangulations of $P^{n-2}$, using a triangulation of $B^{n-2}$
    as in Lemma~\ref{l:quad-ball}(c). Thick lines are the edges of the
    quadrangulation, simplices are depicted in grey (only one is shown
    in the triangulation of $B^3$; note how the quadrangulation
    induces a $K_{1,3}$ on it). Dashed lines represent visibility. We
    write $v'$ for the antipode of the vertex $v$.}
  \label{fig:complete}
\end{figure}

\begin{lemma}\label{l:quad-ball}
  For a graph $G$, consider the following statements:
  \begin{enumerate}[\quad(a)]
  \item \label{l:quad-ball:1} $G$ is a non-bipartite quadrangulation of $P^n$.
  \item \label{l:quad-ball:2} There is a symmetric triangulation $T$
    of $S^n$ such that no simplex of $T$ contains antipodal vertices,
    and there is a proper antisymmetric $2$-colouring of $T$ such that $G$ is
    obtained from the associated graph by identifying all pairs of antipodal
    vertices.
  \item \label{l:quad-ball:3} There is a boundary-symmetric triangulation
    $T'$ of $B^n$ such that no simplex of $T'$ contains antipodal boundary
    vertices, and there is a proper boundary-antisymmetric $2$-colouring of $T'$
    such that $G$ is obtained from the associated graph by identifying
    all pairs of antipodal boundary vertices.
  \end{enumerate}
  Statements (a) and (b) are equivalent and are implied by statement
  (c).
\end{lemma}
\begin{proof}
  We start with the implication
  (\ref{l:quad-ball:1})$\implies$(\ref{l:quad-ball:2}).  If $G$ is a
  quadrangulation of $P^n$, then by definition there exists a
  generalised simplicial complex $K$ such that $\|K\|$ is homeomorphic
  to $P^n$, $G$ is a subgraph of $K^{(1)}$ and every maximal simplex
  of $K$ induces a complete bipartite subgraph of $G$ with at least
  one edge.  It is well known (cf.~\cite{Mun84}) that $P^n$ is the
  quotient space $S^n/\nu$ where $\nu$ is the antipodal action on
  $S^n$, and the projection of $\nu$ is a covering map $p:S^n \to
  S^n/\nu$. Therefore, there is a corresponding simplicial covering
  map $q:T \to K$, where $(T,\xi)$ is a free generalised
  simplicial $\mathbb Z_2$-complex such that $\|T\|\cong S^n$, and the
  homeomorphism induces a centrally symmetric generalised simplicial
  complex structure on $S^n$, with the $\ZZ_2$-action given by the
  antipodal map. The graph $\tilde G=q^{-1}(G)$ is easily seen to be a
  quadrangulation of $T$. Since all cycles in $\tilde G$ are
  0-homologous in the $n$-sphere $\|T\|$, $\tilde G$ is bipartite by
  Lemma~\ref{lem:parity}. Fix a $2$-colouring $c:V(\tilde G) \to
  \{1,2\}$; this defines a colouring of $T$ such that $\tilde G$ is
  its associated graph.  By Lemma~\ref{lem:parity}, $c$ is
  antisymmetric. Suppose that $c$ is not proper. Then there exists an
  $n$-simplex $\sigma \in T$ such that $c(u)=1$ for every $u \in
  \sigma$ (and by antisymmetry, $c(-u)=2$ for all such $u$).
  Therefore, the vertices of $\sigma$ and $-\sigma$ form independent
  sets of $\tilde G$.  Consequently, the vertices of the $n$-simplex
  $q(\sigma) \in K$ form an independent set of $G$, contradicting the
  assumption that $G$ is a quadrangulation of $P^n$. Hence $c$ is
  proper, as required. Finally, the fact that no simplex of $T$
  contains antipodal vertices follows immediately from the fact that
  $K=T/\nu$ has no loops.
  
  We now prove
  (\ref{l:quad-ball:2})$\implies$(\ref{l:quad-ball:1}). Let $T$ be a
  triangulation and $c$ a $2$-colouring of $T$ as in
  property~\ref{l:quad-ball:2}, and let $G_c$ be the associated graph
  of $c$. The quotient of the $\ZZ_2$-action on $T$ is a triangulation
  $K$ of $P^n$. The graph $G$, being obtained from $G_c$ by
  identifying antipodal vertices, is a subgraph of the $1$-skeleton of
  $K$.

  We claim that the subgraph $H$ of $G$ induced on the vertices of any
  maximal simplex $\sigma$ of $K$ is a complete bipartite graph with
  at least one edge. Let $\tau$ be a simplex of $T$ mapped to $\sigma$
  by the covering map corresponding to the $\ZZ_2$-action on $T$. Then
  $H$ is isomorphic to the subgraph of $G_c$ induced on the vertices
  of $\tau$; by the definition of the associated graph, $H$ is
  complete bipartite. Moreover, since $c$ is proper, $H$ has at least
  one edge.
      
  Finally, for the implication
  (\ref{l:quad-ball:3})$\implies$(\ref{l:quad-ball:2}), we take two
  copies of $T'$ (say $T'_1$ and $T'_2$), retain the given $2$-colouring
  on $T'_1$, and invert it on $T'_2$. We glue $T'_1$ and $T'_2$
  together by identifying each simplex of the boundary of $T'_1$ with
  the antipode of its copy in $T'_2$. The resulting generalised
  simplicial complex is a symmetric triangulation $T'_{12}$ of
  $S^n$. Furthermore, note that if the simplices being identified are
  vertices, then they have the same colour. It follows that the
  original proper $2$-colouring $c'$ of $T'$ induces a proper
  $2$-colouring $c'_{12}$ of $T'_{12}$. If we identify the antipodal
  vertices in the associated graph of $c'_{12}$, we obtain the same
  graph (namely $G$) as if we identify the antipodal boundary vertices
  in the associated graph of $c'$. Lastly, note that no simplex of
  $T'_{12}$ contains antipodal vertices: otherwise, by the
  construction, there would be a simplex of $T'_1$ containing
  antipodal boundary vertices, contrary to the assumption.
\end{proof}

It is natural to ask whether statement (c) of Lemma~\ref{l:quad-ball}
is equivalent to statements (a) and (b). This question seems to be
open (cf. the discussion in the last paragraph of~\cite{PS05}).

\section{A lower bound on the chromatic number}
\label{sec:lower}

This section is devoted to the proof of the first of our results
mentioned in Section~\ref{sec:introduction}:
\newtheorem*{theorem-main}{Theorem \ref{thm:main}}
\begin{theorem-main}
  If $G$ is a non-bipartite quadrangulation of $P^n$, then $\chi(G)
  \geq n+2$.
\end{theorem-main}
\begin{proof}
  By Lemma~\ref{l:quad-ball}, there is a symmetric
  triangulation $T$ of $S^n$ and an antisymmetric $2$-colouring
  $c:V(T) \to \{1,2\}$ such that $G$ is obtained from the
  associated graph by identifying antipodal pairs of vertices. Define the
  mapping $f:V(T) \to V(B(G))$
  as $f: v \mapsto (v,c(v))$. 

  Let $A$ be the set of vertices of an arbitrary simplex in $T$.  Set
  $A_i=A \cap c^{-1}(i)$; so $f(A)=A_1 \uplus A_2$.  To prove that $f$
  is a simplicial map, it suffices to show that $A_1 \uplus A_2$ is a
  simplex of $B(G)$. Let $A' \in T$ be the vertex set of a maximal
  simplex in $T$ such that $A \subseteq A'$, and define $A'_i = A'\cap
  c^{-1}(i)$ for $i=1,2$. By the definition of quadrangulation, $A_i
  \subseteq A'_i \subseteq \CN(A'_{3-i})\subseteq \CN(A_{3-i})$ and
  $A'_i \neq \emptyset$, where $i=1,2$.  This shows that $A_1 \uplus
  A_2 \in B(G)$, so $f$ is indeed a simplicial map.

  Moreover, if $\xi$ and $\omega$ are the $\mathbb Z_2$-actions on
  $T$ and $B(G)$, respectively, and $v$ is a vertex in $T$,
  then $f(\xi(v))=(v,3-c(v))=\omega(v,c(v))=\omega(f(v))$, so $f\circ
  \xi=\omega \circ f$. This shows that $f$ is a simplicial $\mathbb
  Z_2$-map, which extends naturally to a simplicial $\mathbb Z_2$-map
  $f':\sd(T) \to \sd(B(G))$; note that $\sd(T)$ is a simplicial
  complex. Its affine extension is therefore a continuous
  $\mathbb Z_2$-map $\|f'\|:S^n \to \|B(G)\|$, so
  $\coind(B(G)) \geq \coind(S^n)=n$. Since $G$ clearly has no isolated
  vertices, the result follows by applying Theorem~\ref{thm:lovasz}.
\end{proof}

Notice that in the proof of Theorem~\ref{thm:main}, we have in fact
shown that $\coind(B(G)) \geq n$ for any non-bipartite
quadrangulation $G$ of $P^n$. In conjunction with results of Simonyi and
Tardos~\cite{SimTar06,SimTar07}, this implies that any non-bipartite
quadrangulation $G$ of $P^n$ satisfies the following properties (for
definitions, see~\cite{SimTar06}):
\begin{enumerate}
\item $G$ has local chromatic number at least $\lceil n/2\rceil +2$;
\item when $n$ is even, $G$ has circular chromatic number at least $n+2$;
\item in any proper colouring of $G$, there is a copy of
$K_{\left\lfloor \frac{n+2}2\right\rfloor, \left\lceil \frac{n+2}2\right\rceil}$
in which all vertices receive different colours;
\item in any proper $(n+2)$-colouring of $G$, there is a copy of
$K_{\ell,m}$ in which all vertices receive different colours, for any
$\ell, m \geq 1$ such that $\ell+m=n+2$.
\end{enumerate}
For $n=2$ these facts were shown in~\cite{MST13,DGMVZ05,AHNNO01,SimTar07},
respectively.

\section{Upper bounds on the chromatic number}
\label{sec:upper}

Every non-bipartite quadrangulation of the projective plane is
4-chromatic; that is, the lower bound proved by Youngs~\cite{You96}
(or given by Theorem~\ref{thm:main}) is actually the right value. This
is not the case in higher dimensions; as the following result shows,
the chromatic number of quadrangulations of $P^3$ is unbounded.

\begin{theorem}\label{t:unbounded}
  For all $r\geq 3$, the complete graph $K_{2r+3}$ embeds in $P^3$ as a
  quadrangulation.
\end{theorem}
\begin{proof}
  Consider the cylinder $C$ in $\RR^3$ that is the product of the unit
  circle $S^1$ in the $xy$ plane with the interval $[-1,1]$ on the $z$
  axis. Let $C^+$ denote the top circle of $C$ and $C^-$ its bottom
  circle. Distribute points $x_0,\dots,x_{2r}$ evenly along $C^+$ and
  colour them black. Further, distribute white points
  $y_0,\dots,y_{2r}$ along $C^-$ in such a way that for each $i$, the
  vertical projection of $x_i$ to $C^-$ is antipodal to $y_i$ on
  $C^-$. Let $V=\Setx{x_0,\dots,x_{2r},y_0,\dots,y_{2r}}$. For any
  geometric simplex $\tau$ on $V$, define $\Theta(\tau)$ as the
  geometric simplicial complex on $V$ whose facets are all the images
  of $\tau$ under rotations of $\RR^3$ about the $z$ axis mapping $V$
  to itself.

  For $j = 1,\dots,r-1$, let $\sigma_j$ be the linear simplex
  $[x_0,x_1,y_{r+j},y_{r+j+1}]$, and let $L_j =
  \Theta(\sigma_j)$. (See Figure~\ref{fig:construction} for an
  illustration.) Further, define $Z^+ = \Theta([x_0,x_1])$
  and $Z^- = \Theta([y_0,y_1])$.

  \begin{figure}
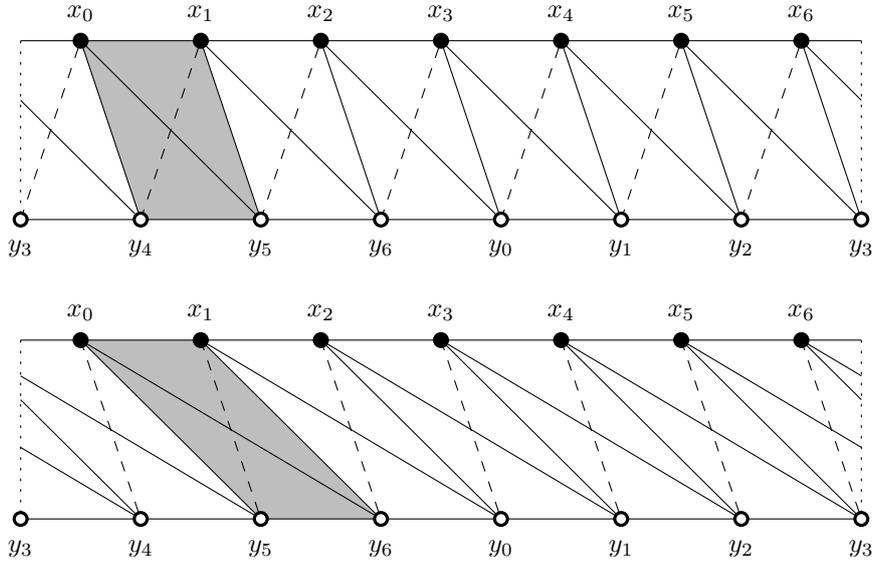

    \centering
    \hf\fig{12}\hf\\[6mm]
    \hf\fig{13}\hf
    \caption{The complexes $L_1$ and $L_2$ in the proof of
      Theorem~\ref{t:unbounded} for $r=3$. The defining 3-simplices
      $[x_0,x_1,y_4,y_5]$ and $[x_0,x_1,y_5,y_6]$ are shown
      grey. Dashed lines represent visibility from the origin. The
      vertical dotted lines are identified.}
    \label{fig:construction}
  \end{figure}

  Let us define the \emph{inner boundary} of any geometric simplicial
  complex in $\RR^3$ as the subcomplex consisting of faces fully
  visible from the origin. We claim that the inner boundary of
  $L_j$ is the complex $L_j^* = \Theta([x_0,x_1,
  y_{r+j+1}]) \cup \Theta([x_0,y_{r+j},y_{r+j+1}])$. (To see this,
  project $\sigma_j$ to the $xy$ plane and note that among the edges
  of $\sigma_j$ with one end white and one end black,
  $[x_0,y_{r+j+1}]$ is the one closest to the $z$ axis;
  cf.~Figure~\ref{fig:projection}. Consequently, the inner boundary
  consists of the $2$-simplices containing it, together with their
  rotational images.)

  \begin{figure}
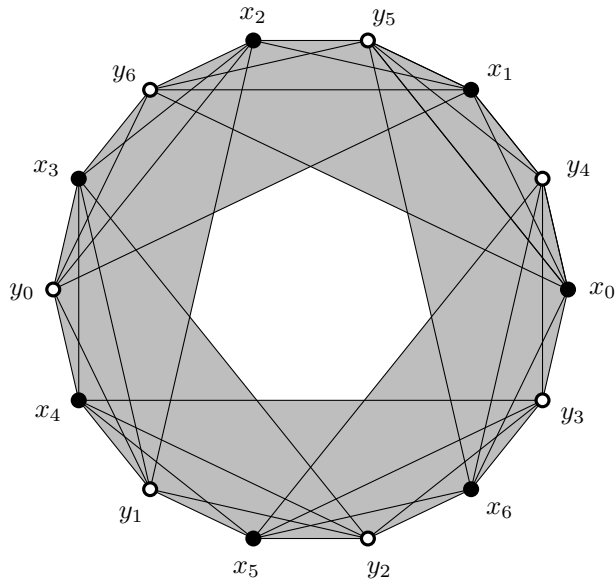

    \centering
    \fig{14}
    \caption{The vertical projection of $L_1\cup L_2$
      (grey) and its $1$-skeleton in the proof of
      Theorem~\ref{t:unbounded} (for $r=3$).}
    \label{fig:projection}
  \end{figure}

  Let $j < \ell$ be distinct integers between $1$ and $r-1$. It is
  easy to check that
  \begin{equation*}
    L_j \cap L_\ell =
    \begin{cases}
      L_j^* & \text{if $\ell = j+1$},\\
      Z^+ \cup Z^- & \text{otherwise}.
    \end{cases}
  \end{equation*}
  Let $L$ be the union of all the complexes $L_j$. By the above, $L$
  is homotopy equivalent to $C$ and its inner boundary is $L_{r-1}^*$.
  
  To make $L$ into a triangulation of the $3$-dimensional
  ball, we do the following:
  \begin{itemize}
  \item place a white vertex $x$ in the centre of $C^+$ and a black
    vertex $y$ in the centre of $C^-$,
  \item add the joins $x * Z^+$ and $y * Z^-$,
  \item place a black vertex at the origin and add its join with
    $L_{r-1}^*\cup (x * Z^+) \cup (y * Z^-)$.
  \end{itemize}
  The resulting complex triangulates the convex hull of $C$, which is
  homeomorphic to $B^3$. It corresponds to a triangulation of $B^3$
  which is symmetric on the boundary; furthermore, the $2$-colouring
  defined above is proper and antisymmetric on the boundary. The
  associated graph, with antipodal boundary vertices identified, is
  easily seen to be the complete graph $K_{2r+3}$.
\end{proof}

We obtain the following result as a corollary of
Theorem~\ref{t:unbounded} for higher dimensions:
\newtheorem*{theorem-upper}{Theorem~\ref{thm:upper}}
\begin{theorem-upper}
  For $n \geq 3$ and $t\geq 5$, the complete graph $K_t$ embeds in
  $P^n$ as a quadrangulation if $t-n$ is even.
\end{theorem-upper}
\begin{proof}
  Using Theorem~\ref{t:unbounded}, embed $K_{t-n+3}$ as a
  quadrangulation in $P^3$. Let $T$ be a $2$-coloured triangulation of
  $S^3$ satisfying the condition of Lemma~\ref{l:quad-ball}(b). Taking
  the $(n-3)$-fold suspension, that is, the $(n-3)$-fold join of $T$
  with $S^0$ (consisting of two points, one black and one white), we
  obtain a triangulation $T'$ of $S^n$, again satisfying the condition
  of Lemma~\ref{l:quad-ball}(b). Moreover, identifying the antipodal
  vertices in the associated graph, we obtain $K_t$.
\end{proof}

While unbounded in general in dimension $n>2$, the chromatic number of
higher-dimensional projective quadrangulations is bounded when the
quadrangulation is `sufficiently fine', as shown by the following
proposition.

\begin{proposition}
\label{prop:borsuk}
Let $G$ be a quadrangulation of $P^n$ and $T$ the corresponding
$2$-coloured symmetric triangulation of $S^n$. If the vertices of $T$
lie on the unit sphere $S^n$ in $\mathbb R^{n+1}$ and adjacent
vertices of different colours are at Euclidean distance less than
$\frac 2{\sqrt{n+3}}$, then $\chi(G)=n+2$.
\end{proposition}

Before proving the proposition, let us recall the following construction
due to Erd\H os and Hajnal~\cite{ErdHaj67}.
The \emph{Borsuk graph} $B(n,\alpha)$ is defined as the (infinite) graph
whose vertices are the points of $\mathbb R^n$ on $S^{n-1}$, and the edges
connect points at Euclidean distance at least $\alpha$, where $0<\alpha<2$.
By the Borsuk-Ulam theorem~\cite{Bor33}, $\chi(G) \geq n+1$ (in fact the two
statements are equivalent, as noted by Lov\'asz~\cite{Lov83}). By using the
standard $(n+1)$-colouring of $S^{n-1}$ (based on the central projection of a
regular $n$-simplex), it may be shown that $B(n,\alpha)$ is $(n+1)$-chromatic
for all $\alpha$ sufficiently large. In particular, Simonyi and
Tardos~\cite{SimTar06} have shown that if $\alpha>2 \sqrt{1-1/(n+2)}$ then
$B(n,\alpha_0)$ is $(n+1)$-chromatic.

\begin{proof}[Proof of Proposition~\ref{prop:borsuk}]
  Consider the symmetric triangulation $T$ and identify the vertices
  of $G$ with their corresponding black vertices in $\|T\|$.  This
  gives an embedding of $G$ in $\mathbb R^{n+1}$ with all vertices on
  $S^n$, such that adjacent vertices are `nearly antipodal'. More
  precisely, the Euclidean distance between them is greater than
  $\alpha_0$, where
  \[
    \alpha_0=\sqrt{4-\left(\tfrac 2{\sqrt{n+3}}\right)^2}=2\sqrt{1-\tfrac 1{n+3}}.
  \]
  In particular, $G$ is a subgraph of the Borsuk graph $B(n+1,\alpha)$,
  where $\alpha>\alpha_0$. As shown by Simonyi and Tardos~\cite{SimTar06},
  $\chi(B(n+1,\alpha)) \leq n+2$, so a fortiori $\chi(G) \leq n+2$.
\end{proof}

\section{Application to Kneser and Mycielski graphs}
\label{sec:mycielski}

Recall that for given $n,k\geq 1$, where $n > 2k$, the vertices of
the \emph{Kneser graph} $KG(n,k)$ are the $k$-element subsets of $[n]
= \Setx{1,\dots,n}$, and edges join pairs of subsets that are
disjoint.

Schrijver~\cite{Sch78} characterised a family of vertex-critical
subgraphs of Kneser graphs. Viewing $[n]$ as the vertex set of the
$n$-cycle $C_n$ (with edges $12$, $23$, \dots, $(n-1)n$, $n1$), let
$\mathcal I_k(C_n)$ be the collection of all $k$-element independent
subsets of $C_n$. The \emph{Schrijver graph} $SG(n,k)$ is defined as
the induced subgraph of $KG(n,k)$ on $\mathcal I_k(C_n)$. As shown
in~\cite{Sch78}, $SG(n,k)$ is $(n-2k+2)$-chromatic and
vertex-critical. 

In this section, we prove the following result, stated in
Section~\ref{sec:introduction}:
\newtheorem*{theorem-schrijver}{Theorem \ref{thm:schrijver}}
\begin{theorem-schrijver}
  Let $n > 2k$ and $k \geq 1$. There exists a non-bipartite
  quadrangulation of $P^{n-2k}$ that is homomorphic to $SG(n,k)$.
\end{theorem-schrijver}
In view of Theorem~\ref{thm:main}, Theorem~\ref{thm:schrijver} gives
an alternative proof that $\chi(SG(n,k)) \geq n-2k+2$.

To obtain the quadrangulation in Theorem~\ref{thm:schrijver}, we
utilise the generalised Mycielski construction of Gy\'{a}rf\'{a}s,
Jensen and Stiebitz~\cite{GJT04} (see also~\cite[p.~133]{Mat03}). Let
$r\geq 1$ and let $G$ be a graph with vertex set $V$. The
\emph{generalised Mycielskian} $M_r(G)$ of $G$ in defined as follows:
\begin{enumerate}[\quad (M1)]
\item the vertex set of $M_r(G)$ is $\Setx{z}\cup (V\times [r])$,
\item for $i = 2,\dots,r$ and adjacent vertices $v,w$ of $G$, the
  vertices $(v,i)$ and $(w,i-1)$ of $M_r(G)$ are adjacent,
\item $M_r(G)$ contains a copy of $G$ on $V\times\Setx r$,
\item the vertex $z$ (which we will refer to as the \emph{universal}
  vertex) is adjacent to all vertices in $V\times\Setx{1}$.
\end{enumerate}
(We note that the original definition was given for $r\geq 2$; we
extend it to the case $r=1$, which consists in adding the universal
vertex.)

It is not difficult to see that for all $r\geq 1$, the chromatic
number of $M_r(G)$ is at most $\chi(G)+1$. While there are graphs for
which the inequality is strict, it follows from the results
of~\cite{GJT04} that equality holds for graphs obtained from an
odd cycle by a finite number of iterations of the generalised
Mycielskian $M_r(\cdot)$.

For $n\geq 3$, the \emph{Mycielski graph} $M_n$ is defined as the
graph obtained from the 5-cycle by the $(n-3)$-fold iteration of the
operation $M_2(\cdot)$. In particular, $M_3$ is the 5-cycle. It is
well known~\cite{Myc55} that for each $n$, $M_n$ is triangle-free and
$n$-chromatic. The following theorem gives a new geometrical intuition
for this fact:

\begin{theorem}\label{thm:mycielski}
  If $G$ is a nonbipartite quadrangulation of $P^n$ and $r\geq 1$,
  then $M_r(G)$ is a nonbipartite quadrangulation of $P^{n+1}$. In
  particular, the Mycielski graph $M_n$ and the complete graph $K_n$
  embed as nonbipartite quadrangulations in $P^{n-2}$.
\end{theorem}
\begin{proof}
  By Lemma~\ref{l:quad-ball}, let $T$ be an antipodally
  symmetric triangulation of $S^n$ with an antisymmetric $2$-colouring
  such that $G$ is obtained from the associated graph by identifying
  antipodal vertices. 

  We will extend $T$ to a triangulation $L$ of the
  ball $B^{n+1}$ (with a proper $2$-colouring) whose associated graph,
  with antipodal boundary vertices identified, is $M_r(G)$. The
  construction can be viewed as a counterpart of the Mycielski
  construction on the level of simplicial complexes.

  To avoid excessive formalism, we will assume that $T$ is a
  geometric simplicial complex with a realization $\|T\|$ in
  $\RR^{n+1}$ where all the vertices lie on the unit $n$-sphere
  $S^n$. The general case (where we have a generalised simplicial
  complex and the simplices are not necessarily linear) follows using
  the same idea.

  Let us fix an arbitrary linear order $\preceq$ on the vertices of
  $G$ (the \emph{order of precedence}).  We construct a sequence
  $L_r,\dots, L_1, L_0$ of geometric complexes with proper
  $2$-colourings starting with $L_r = T$, ending with $L_0 = L$, and
  such that for each $i = r,\dots,1$, $L_i$ is a subcomplex of
  $L_{i-1}$ (as a $2$-coloured complex).

  The complex $L_r$ has two antipodal vertices of different
  colours for each vertex $v\in V(G)$; let us call the white one $v^r$
  and the black one $v^{r+1}$. For $1\leq i \leq r-1$, the set
  $V(L_i) - V(L_{i+1})$ is denoted by $V^i$ and
  consists of vertices $v^i$, where $v\in V(G)$. We define $V^r =
  V(L_r)$ and $V^0 = \Setx z$, where $z$ is a special vertex. For each
  vertex of $V^i$ ($0\leq i \leq r$), the integer $i$ is its
  \emph{level}.

  For each $i = r-1,\dots,1$, we extend $L_{i+1}$ to $L_i$ as
  follows. (See the illustration in Figures~\ref{fig:mycielski1} and
  \ref{fig:mycielski2}.) Define the set of \emph{active} vertices as
  $V^{i+2} \cup V^{i+1}$. For each vertex $v$ of $G$ in the order of
  precedence, do the following:
  \begin{itemize}
  \item add the vertex $v^i$, embed it in the open segment from
    $v^{i+2}$ to the origin and assign it the colour of $v^{i+2}$,
  \item for each simplex $\sigma$ containing $v^{i+2}$ and consisting
    of active vertices, add the simplex $\sigma\cup\Setx{v^i}$ (with
    the linear embedding),
  \item mark $v^i$ active and $v^{i+2}$ inactive.
  \end{itemize}

  Note that the colour of each vertex of $L$ only depends on its
  level, and it alternates as the level changes from $r$ to $0$.

  As a final step, construct the complex $L_0 = L$:
  \begin{itemize}
  \item add the vertex $z$, placing it at the origin, and assign it
    the colour given to the vertices in $V^2$,
  \item for each simplex $\sigma$ of $L_1[V_2\cup V_1]$, add the
    simplex $\sigma\cup\Setx{z}$ to $L_0$.
  \end{itemize}

  \begin{figure}
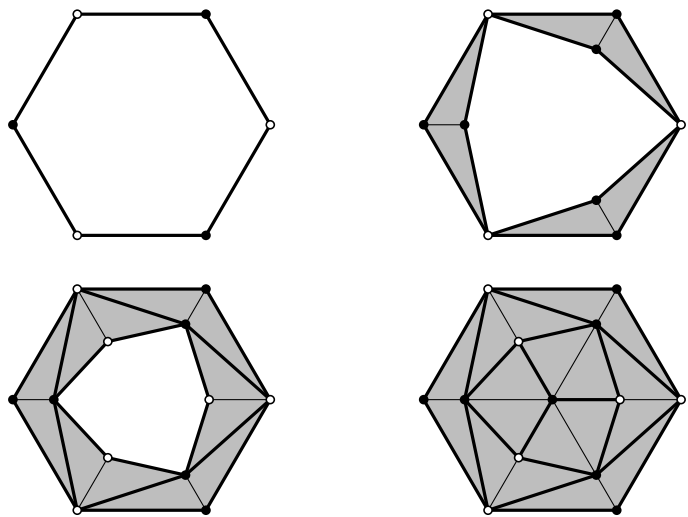

    \centering
    \hf\fig5\hf\fig6\hf\\[5mm]
    \hf\fig7\hf\fig8\hf
    \caption{The complex $T=L_3$ and the complexes
      $L_2$, $L_1$ and $L_0 = L$
      in the proof of Theorem~\ref{thm:mycielski} (for $n=1$ and
      $r=3$). The thick edges are the edges of the associated graph,
      the grey regions represent $2$-simplices.}
    \label{fig:mycielski1}
  \end{figure}
 
  \begin{figure}
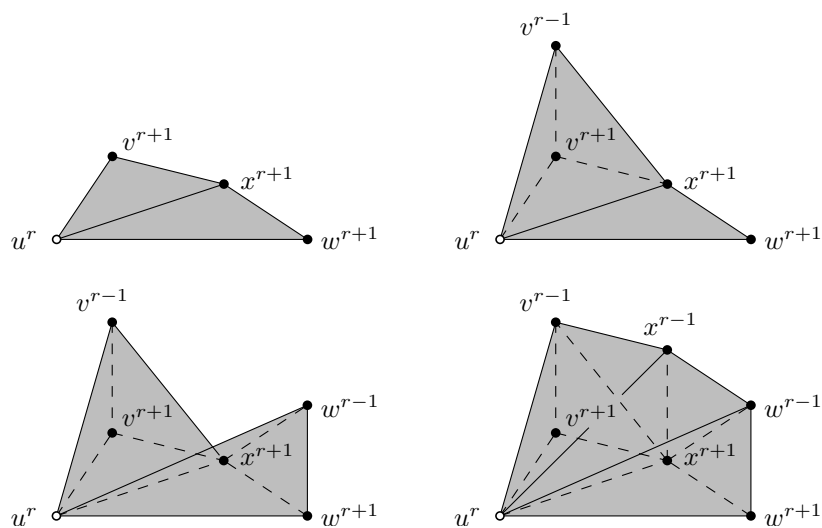

    \centering
    \hf\fig1\hf\fig2\hf\\[5mm]
    \hf\fig3\hf\fig4\hf
    \caption{Adding $3$-simplices in the construction of $L_{r-1}$ in
      the proof of Theorem~\ref{thm:mycielski} (for $n=2$, assuming
      $v\preceq w \preceq x$, where $v,w,x\in V(G)$). Dashed lines
      indicate visibility, the grey regions represent simplices of
      dimension $2$ and $3$.}
    \label{fig:mycielski2}
  \end{figure}
 
  Clearly, $L$ is an $(n+1)$-dimensional complex containing
  $T$ as a subcomplex. We claim that $L$
  triangulates the $(n+1)$-dimensional ball $B^{n+1}$. Indeed, each
  complex $L_i$ in the sequence is obtained from the preceding one by
  `thickening' the simplices of $L_{i+1}[V^{i+2}\cup V^{i+1}]$ at the
  vertices of $V^{i+2}$, in the direction to the origin. After a
  series of these steps, the thickened sphere $L_1$ is filled
  in by inserting the vertex $z$ and joining it to all the simplices
  forming the `interior boundary' of $L_1$.

  Let $\tilde G$ be the graph obtained from the associated graph of
  $L$ by identifying antipodal boundary vertices. We show
  that $\tilde G = M_r(G)$. To this end, we check conditions
  (M1)--(M4) in the definition of $M_r(G)$.

  The vertex set of $\tilde G$ is $\Setx z \cup V(G) \times [r]$,
  since the identification of antipodal vertices of $T$
  identifies each vertex at level $r+1$ with a vertex at level
  $r$. (If the identification involves vertices, say, $v^{r+1}$ and
  $v^r$, we continue to call the resulting vertex $v^r$.) Furthermore,
  the induced subgraph of $\tilde G$ on the set of vertices arising
  from the identification is $G$ by the assumption on $T$. So
  far, we have checked conditions (M1) and (M3).

  To prove (M2) and (M4), recall that the colours of vertices of
  $L$ alternate with respect to their level. Since the levels
  of vertices in each simplex of $L$ differ by at most two,
  all the edges of $\tilde G$ come from $1$-simplices of $L$ of
  the form $v^iw^{i-1}$, where $v,w\in V(G)$ and $1\leq i \leq
  r+1$. It is easy to prove by induction that for each $i$, $2\leq i
  \leq r$, and all $u,v\in V(G)$, $L$ contains the $1$-simplex
  $[v^i,w^{i-1}]$ if and only if $v$ and $w$ are neighbours in
  $G$. This implies property (M2). Property (M4) follows once we note
  that $L$ contains $1$-simplices $[z,v^1]$ for all $v\in
  V(G)$, they are not monochromatic, and $z$ does not form a $1$-simplex
  with any vertex at level greater than $2$.

  The proof that $\tilde G = M_r(G)$ is complete. The statement about
  the Mycielski graphs and the complete graphs follows from the above by
  induction (and the observation that the $5$-cycle and $K_3$ embed in
  $P^1$ as quadrangulations).
\end{proof}

Theorem~\ref{thm:schrijver} will be derived from the following lemma:

\begin{lemma}\label{l:homo}
  If $k\geq 1$ and $n > 2k+1$, then the graph $M_k(SG(n-1,k))$ is
  homomorphic to $SG(n,k)$. 
\end{lemma}
\begin{proof}
  We will explicitly describe a homomorphism $f$ from $M_k(SG(n-1,k))$
  to $SG(n,k)$. Let $I$ be a vertex of $SG(n-1,k)$ and let $\cp I
  1,\dots,\cp I k$ be its copies in $M_k(SG(n-1,k))$. Furthermore,
  let $Z$ be the universal vertex of $M_k(SG(n-1,k))$.

  Suppose that $I = \Setx{a_1,\dots,a_k}$, where $a_1 < \dots <
  a_k$. We first define $f(\cpx I {k - 2i})$, where $0 \leq i <
  k/2$. The image is obtained from $I$ by replacing the first $i$
  elements by the $i$ least odd numbers, and replacing the last $i$
  elements by the arithmetic progression of length $i$ and step 2
  ending with $n-1$. In symbols,
  \begin{align*}
    f(\cpx I {k-2i}) = &\{1,3,\dots,2i-1,\\
      &a_{i+1},a_{i+2},\dots,a_{k-i},\\
      &n-2i+1,n-2i+3,\dots,n-1\}.
  \end{align*}
  In particular, $\cp I k$ is mapped to $I$. Since, clearly, $a_{i+1}
  \geq 2i+1$ and $a_{k-i} \leq n-2i-1$, we find that $f(\cpx I
  {k-2i})$ is a vertex of $SG(n,k)$.

  Next, we define $f(\cpx I {k-2i-1})$, where $0 \leq i \leq k/2-1$. In
  the case that $a_1 > 1$, we set
  \begin{align*}
    f(\cpx I {k-2i-1}) = &\{2,4,\dots,2i,\\
      &a_{i+1},a_{i+2},\dots,a_{k-i-1},\\
      &n-2i,n-2i+2,\dots,n\}.
  \end{align*}
  The definition for the case $a_1 = 1$ is almost the same, except
  that the element $a_{i+1}$ is replaced by $a_{k-i}$.

  To complete the definition, it remains to set
  \begin{equation*}
    f(Z) =
    \begin{cases}
      \Setx{1,3,\dots,k-1,n-k+1,n-k+3,\dots,n-1} & \text{if $k$ is even,}\\
      \Setx{2,4,\dots,k-1,n-k+1,n-k+3,\dots,n-2,n} & \text{if $k$ is odd.}
    \end{cases}
  \end{equation*}

  Note that the image of $f$ is contained in the vertex set of
  $SG(n,k)$. To verify that $f$ is a homomorphism, we consider a pair
  of adjacent vertices of $M_k(SG(n-1,k))$ and show that their images
  are disjoint (that is, adjacent in $SG(n,k)$).

  First, consider vertices $\cp I {k-2i}$ and $\cp J {k-2i-1}$, where
  $I=\Setx{a_1,\dots,a_k}$ and $J=\Setx{b_1,\dots,b_k}$ are disjoint
  and $0 \leq i < k/2-1$. Suppose that $x\in f(\cpx I {k-2i})\cap
  f(\cpx J {k-2i-1})$. It follows that $x \geq 2i+1$ (as all the
  smaller elements of $f(\cpx I {k-2i})$ are odd and all the smaller
  elements of $f(\cpx J {k-2i-1})$ are even). Similarly, $x\leq
  n-2i-1$. Then, however, the definition implies that $x\in I\cap J$,
  a contradiction.

  A similar argument works for the vertices $\cp I {k-2i-1}$ and $\cp
  J {k-2i-2}$ ($0 \leq i \leq k/2-2$). Thus, the only remaining case
  is the pair $Z$ and $\cp I 1$. Suppose that $k$ is even. By the
  definition,
  \begin{align*}
    f(\cpx I 1) &= \Setx{2,4,\dots,k-2,a_{k/2},n-k+2,n-k+4,\dots,n},\\
    f(Z) &= \Setx{1,3,\dots,k-1,n-k+1,n-k+3,\dots,n-1}.
  \end{align*}
  Since $a_{k/2}\notin\Setx{k-1,n-k+1}$, we find $f(\cpx I 1)\cap f(Z)
  = \emptyset$. An analogous argument for odd $k$ completes the proof
  that $f$ is a homomorphism.
\end{proof}

\begin{proof}[Proof of Theorem~\ref{thm:schrijver}]
  By Lemma~\ref{l:homo}, $M_k(SG(n-1,k))$ is homomorphic to
  $SG(n,k)$. This implies that $M_k(M_k(SG(n-2,k))$ is homomorphic to
  $SG(n,k)$ since, in general, any graph homomorphism from $H$ to $H'$
  determines a homomorphism from $M_k(H)$ to $M_k(H')$. 

  Continuing, we find that the graph
  \begin{equation}\label{eq:iterated}
    M_k(M_k(\ldots M_k(SG(2k+1,k)) \ldots )),
  \end{equation}
  where $M_k(\cdot)$ is applied $n-2k-1$ times, is homomorphic to
  $SG(n,k)$. Since $SG(2k+1,k)$ quadrangulates $P^1$,
  Theorem~\ref{thm:mycielski} implies that \eqref{eq:iterated} is a
  non-bipartite quadrangulation of $P^{n-2k}$.
\end{proof}

\section{Conclusion}
\label{sec:conclusion}

We conclude the paper with some open questions. In relation to
Schrij\-ver graphs, it appears likely that Theorem~\ref{thm:schrijver}
can be strengthened as follows:
\begin{conjecture}
  For $n > 2k \geq 2$, the Schrijver graph $SG(n,k)$ contains a
  non-bipartite quadrangulation of $P^{n-2k}$ as a spanning subgraph. 
\end{conjecture}

The result of Youngs concerning quadrangulations of $P^2$ was extended
to arbitrary non-orientable surfaces by Archdeacon et al.~\cite{AHNNO01} and
by Mohar and Seymour~\cite{MohSey02} using the notion of odd quadrangulation. 

\begin{problem}
  Is a similar generalisation possible in the higher-dimensional case?
\end{problem}

An interesting question posed by a referee is whether the geometric condition
in Proposition~\ref{prop:borsuk} could be replaced by a condition on the odd
girth of the quadrangulation (the length of the shortest odd cycle). In
particular:

\begin{problem}
  Is there a constant $g$ such that every quadrangulation
  of $P^n$ of odd girth at least $g$ is $(n+2)$-chromatic?
\end{problem}

\section*{Acknowledgments}

We thank the anonymous referees for their helpful comments and
suggestions.

\bibliographystyle{plain}
\bibliography{projective}

\end{document}